\documentclass[12pt]{amsart}
\usepackage{amsmath,amssymb,amsfonts,amsthm,amsopn}
\usepackage{graphics}

\newcommand{\tfa}{time-frequency analysis}

\newcommand{\ft}{Fourier transform}
\newcommand{\stft}{short-time Fourier transform}

\newcommand{\tf}{time-frequency}

\newcommand{\tfs}{time-frequency shift}

\newcommand{\aw}{Anti-Wick}
\newcommand{\aaf}{A_a^{\varphi_1,\varphi_2}}

\newcommand{\modsp}{modulation space}

\newtheorem{tm}{Theorem}[section]
\newtheorem{lemma}[tm]{Lemma}

\newtheorem{theorem}{Theorem}[section]
\newtheorem{corollary}[theorem]{Corollary}

\newtheorem{proposition}[theorem]{Proposition}
\newtheorem{remark}[theorem]{Remark}

\newcommand{\beqa}{\begin{eqnarray*}}
\newcommand{\eeqa}{\end{eqnarray*}}

\newcommand{\field}[1]{\mathbb{#1}}
\newcommand{\bR}{\field{R}}        
\def\la{\lambda}

 \def\cF{\mathcal{F}}              
 \def\cS{\mathcal{S}}

 \def\cU{\mathcal{U}}

\def\gaw{A_a^{\f_1,\f_2}}

\def\vgf{V_gf}

\def\rd{\bR^d}

\def\rdd{{\bR^{2d}}}

\def\lrd{L^2(\rd)}
\def\lrdd{L^2(\rdd)}

\def\Qs{{Q_s}}
\def\intrd{\int_{\rd}}
\def\intrdd{\int_{\rdd}}

\def\R{\right)}

\def\<{\left<}
\def\>{\right>}

\def\mv1{M_v^1}

\def\Mmpq{M_m^{p,q}}
\def\phas{(x,\xi )}

\def\o{\xi}

\def\z{\zeta}

\def\R{\mathbb{R}}
\def\Ren{\mathbb{R}^d}
\def\Renn{\mathbb{R}^{2d}}

\def\Qs{{Q_s}}

\def\sch{\mathcal{S}}

\def\Fur{\mathcal{F}}

\def\f{\varphi}

\def\gaw{A_a^{\f_1,\f_2}}

\def\Sn2{S_{2}(L^{2}(\Ren))}
\def\S1{S_{1}(L^{2}(\Ren))}
\def\sig00{\sigma_{0,0}}

\def\la{\langle}
\def\ra{\rangle}

\begin{document}


\title[Sharp integral bounds for  Wigner distributions]{Sharp integral bounds for Wigner distributions}
\author{Elena Cordero}
\address{Dipartimento di Matematica, Universit\`a di Torino, Dipartimento di
Matematica, via Carlo Alberto 10, 10123 Torino, Italy}
\email{elena.cordero@unito.it}
\thanks{}
\author{Fabio Nicola}
\address{Dipartimento di Scienze Matematiche, Politecnico di Torino, corso
Duca degli Abruzzi 24, 10129 Torino, Italy}
\email{fabio.nicola@polito.it}
\thanks{}

\address{Dipartimento di Matematica,
Universit\`a di Torino, via Carlo Alberto 10, 10123 Torino, Italy}
\address{Dipartimento di Scienze Matematiche,
Politecnico di Torino, corso Duca degli Abruzzi 24, 10129 Torino,
Italy}

\email{elena.cordero@unito.it}
\email{fabio.nicola@polito.it}

\date{}

\begin{abstract}
The  cross-Wigner distribution $W(f,g)$ of two functions or temperate distributions $f,g$ is a fundamental tool in quantum mechanics and in signal analysis. Usually, in applications in time-frequency analysis $f$ and $g$ belong to some modulation space and it is  important to know  which modulation spaces $W(f,g)$ belongs to. Although several particular sufficient conditions have been appeared in this connection, the general problem remains open. In the present paper we solve completely this issue by providing the full range of modulation spaces in which the continuity of the cross-Wigner distribution $W(f,g)$ holds, as a function of $f,g$. The case of weighted modulation spaces is also considered. The consequences of our results are manifold:  new bounds for the short-time Fourier transform and the  ambiguity function, boundedness results for pseudodifferential (in particular, localization) operators and properties of the Cohen class.
\end{abstract}

\maketitle

\section{Introduction}

The (cross-)Wigner distribution was first introduced in physics to account for quantum corrections to classical statistical mechanics in 1932 by  Wigner \cite{Wigner32} and in 1948 it was  proposed in signal analysis by Ville \cite{Ville}. This is why the Wigner distribution is also called  Wigner-Ville distribution. Nowadays it can be considered one of the most important time-frequency representations, second only to the spectrogram, and it is one of the most commonly used quasiprobability distribution in quantum mechanics \cite{Birkbis,book}. 

Given two functions $f_1,f_2\in
L^2(\Ren)$, the {\it cross-Wigner distribution} $W(f_1,f_1)$  is defined to be
\begin{equation}
\label{eq3232}
W(f_1,f_2)(x,\o)=\int f_1(x+\frac{t}2)\overline{f_2(x-\frac{t}2)} e^{-2\pi
	i\o t}\,dt.
\end{equation}
The quadratic expression $Wf = W(f,f)$ is called the Wigner
distribution of $f$. 

An important issue related to such a distribution is the continuity of the map $(f_1,f_2)\mapsto W(f_1,f_2)$ in the relevant Banach spaces. The basic result in this connection is the easily verified equality
\[
\|W(f_1,f_2)\|_{L^2(\rdd)}=\|f_1\|_{L^2(\rd)}\|f_2\|_{L^2(\rd)}.
\]
Beside $L^2$, the time-frequency concentration of signals is often measured by the so-called modulation space norm $M^{p,q}_{m}$, $1\leq p,q\leq\infty$,  for a suitable weight function $m$ (cf. \cite{feichtinger80,feichtinger83,book} and Section $2$ below).
In short, these spaces are defined as follows. For a fixed non-zero $g \in \cS (\rd )$, the \stft\ (STFT) of $f \in
\cS ' (\rd ) $ with respect to the window $g$ is given by
\begin{equation}
\label{eqi2}
V_gf(x,\o)=\int_{\Ren}
f(t)\, {\overline {g(t-x)}} \, e^{-2\pi i\o t}\,dt\, .
\end{equation}
Then the space $M^{p,q}_{m}(\rd)$ is defined by
\[
M^{p,q}_{m}(\rd)=\{f\in \cS'(\rd):\ V_g f\in L^{p,q}_{m}(\rdd)\}
\]
endowed with the obvious norm. Here $L^{p,q}_{m}(\rdd)$ are mixed-norm weighted Lebesgue spaces in $\rdd$; see Section 2 below for precise definitions.\par

Both the STFT  $\vgf $ and the cross-Wigner distribution  $W(f,g)$ are defined on many pairs
of Banach spaces.  For example, they both map $L^2(\rd ) \times
L^2(\rd )$ into $\lrdd $ and  $\sch(\Ren)\times\sch(\Ren)$ into
$\sch(\Renn)$ and  can be extended  to a map from
$\sch'(\Ren)\times\sch'(\Ren)$ into $\sch'(\Renn)$.\par
In this paper we will mainly work with the polynomial weights
\begin{equation}
v_s(z)=\la z\ra^s=(1+|z|^2)^{\frac s2},\quad z\in\rdd,\quad s\in\bR.
\end{equation} 
For $w=(z,\zeta)\in\bR^{4d}$, we write $(1\otimes v_s)(w)=v_s(\zeta)$. Now, the problem addressed in this paper is to provide the full range of exponents $ p_1,p_2,q_1,q_2,p,q\in [1,\infty]$ such that 
\[
\|W(f_1,f_2)\|_{M^{p,q}_{1\otimes v_s}}\lesssim \|f_1\|_{M^{p_1,q_1}_{v_s}} \|f_2\|_{M^{p_2,q_2}_{v_s}}.
\]
 These estimates were proved in  \cite[Theorem 4.2]{Toftweight} (cf.\ also \cite[Theorem 4.1]{toft1} for \modsp s without weights) under the conditions 
\[
p\leq p_i,q_i\leq q, \quad \ i=1,2\]
 and
   \begin{equation}\label{Wigindex}
    \frac1{p_1}+\frac1{p_2}=\frac1{q_1}+\frac1{q_2}=\frac1{p}+\frac1{q}.
   \end{equation}
However, it is not clear whether these conditions are necessary as well. \par
Our main result shows that the sufficient conditions can be widened and such  extension is sharp.
\begin{theorem} \label{T1} Assume $p_i,q_i,p,q\in [1,\infty]$, $s\in \bR$, such that
	\begin{equation}\label{WIR}
	p_i,q_i\leq q,  \ \quad i=1,2
	\end{equation}
	and that
	\begin{equation}\label{Wigindexsharp}
	\frac1{p_1}+\frac1{p_2}\geq \frac1{p}+\frac1{q},\quad \frac1{q_1}+\frac1{q_2} \geq \frac1{p}+\frac1{q}.
	\end{equation}
	Then,
	if $f_1\in M^{p_1,q_1}_{v_{|s|}}(\Ren)$ and
	$f_2\in M^{p_2,q_2}_{v_s}(\Ren)$ we have  $W(f_1,f_2)\in
	M^{p,q}_{1\otimes v_s}(\Renn)$, and
	\begin{equation}\label{wigest}
	\| W(f_1,f_2)\|_{M^{p,q}_{1\otimes v_s}}\lesssim
	\|f_1\|_{M^{p_1,q_1}_{v_{|s|}}}\| f_2\|_{M^{p_2,q_2}_{v_s}}.
	\end{equation}
	
	\par
	Viceversa, assume that there exists a constant $C>0$ such that
	\begin{equation}\label{Wigestsharp}
	\|W(f_1,f_2)\|_{M^{p,q}}\leq C \|f_1\|_{M^{p_1,q_1}} \|f_2\|_{M^{p_2,q_2}},\quad \forall f_1,f_2\in\cS(\rdd).
	\end{equation}
	Then \eqref{WIR} and \eqref{Wigindexsharp} must hold.
\end{theorem}

The remarkable fact of this result, in our opinion, is that the conditions \eqref{WIR} and \eqref{Wigindexsharp} turn out to be necessary too. \par\medskip

The consequences of this are manifold. First, in the framework of signal analysis and time-frequency representations, we obtain new estimates for  the short-time Fourier transform  $V_{f_1} f_2$ and the ambiguity function $A(f_1,f_2)$ (see Section 2 for definitions). In particular, we recapture the sharp Lieb's bounds in \cite[Theorem 1]{Lieb}
$$\|A(f_1,f_2)\|_{L^q}\lesssim \|f_1\|_{L^2}\|f_2\|_{L^2},$$
valid for  $q\geq 2$; we also refer to \cite{cordero1,cordero2} for related estimates for the short-time Fourier transform and \cite{CDNACHA2016} for the strictly related Born-Jordan distribution. \par

Secondly, we easily provide new boundedness results for pseudodifferential operators (in particular, localization operators) with symbols in modulation spaces. Let us mention that
the study of pseudodifferential operators in the context of modulation spaces
has been pursued by many authors. The earliest works are due to Sj\"ostrand \cite{Sjostrand1} and Tachizawa \cite{Tachizawa1}. In the former work  pseudodifferential operators with symbols in the modulation space $M^{\infty,1}$ (also called Sj\"ostrand's class) where investigated. Later, sufficient and some necessary boundedness conditions where investigate by Gr\"ochenig and Heil \cite{GH99,GH04} and Labate \cite{Labate1,Labate2}. Since the year 2003 until today the contributions on this topic are so multiplied that there is hard to mention them all. Let us just recall some of them \cite{benyi,benyi2,CTWJFA2014,Czaja,grstrohmer,KobSug2011,ptt2,ptt1,strohmer,sugitomita2,toft1,Toftweight}.

Every continuous operator from $\cS (\rd )$ to $\cS ' (\rd )$ can be
represented as a pseudodifferential operator in the  Weyl form $L_\sigma$  and the connection with the cross-Wigner distribution is provided  by
 \begin{equation}\label{equiv1}
 \la L_\sigma f,g\ra=\la \sigma,W(g,f)\ra,\quad\quad f,g\in\sch(\Ren).
 \end{equation}
 By using this formula we can translate the boundedness results for the cross-Wigner distribution in Theorem \ref{T1} to boundedness results for Weyl operators. 

Pseudodifferential operators of great interest in signal analysis are the so-called  localization operators $\gaw$ (see Section \ref{psdo}), which can be represented as  Weyl operators as follows (cf. \cite{BCG2004,CG02,Nenadmed2016})
 \begin{equation}\label{WA}\gaw=L_{a\ast W(\f_2,\f_1)}
  \end{equation}  so that the Weyl symbol of the localization operator $\aaf$
is given by
\begin{equation}\label{eq2}
\sigma = a\ast W(\f_2,\f_1)\, .
\end{equation}
Using this representation of localization operators as Weyl operators and using Theorem \ref{T1} we are able to obtain new boundedness results for localization operators, see Theorem \ref{tempbound} in Section \ref{psdo} below.

Finally, another application of Theorem \ref{T1} is the investigation of the time-frequency properties of the Cohen class, introduced by
Cohen in \cite{Cohen1}. This class consists of elements of the type 
\begin{equation}\label{Cohenclass}
M(f,f)=W(f,f)\ast \sigma
\end{equation}%
where $\sigma \in \mathcal{S}^{\prime }({\mathbb{R}^{2d}})$ is called the
Cohen kernel. When $\sigma =\delta $, then $M(f,f)=W(f,f)$ and we come back
to the Wigner distribution. For other choices of kernels we recapture  the Born-Jordan distribution \cite{BJinv,CDNACHA2016,CGNtfa} or the $\tau $-Wigner
distributions  $W_{\tau }(f,f)$ \cite[Proposition 5.6]{bogetal}.
In this framework we have the following result.
\begin{theorem}\label{cohenbound}
	Assume $s\geq0$, $p_1,q_1,p,q \in [1,\infty]$ such that
	\begin{equation}\label{Ch}
 2\min \{ \frac 1{p_1},\frac 1{q_1}\}\geq \frac 1p +\frac 1 q.
	\end{equation}
	Consider a  Cohen kernel  $\sigma \in M^{1,\infty}(\rdd)$. If $f\in M^{p_1,q_1}_{v_s}(\rd)$, then the Cohen distribution  $M(f,f)$ is in $ M^{p,q}_{1\otimes v_{s}}(\rdd)$, with
	\begin{equation}\label{cohenM}
	\|M(f,f)\|_{M^{p,q}_{1\otimes v_{s}}(\rdd)}\lesssim \|\sigma\|_{M^{1,\infty}(\rdd)} \|f\|^2_{M^{p_1,q_1}_{v_s}(\rd)}.
	\end{equation}
	\end{theorem} 
In particular, the $\tau$-kernels and the Born-Jordan kernels enjoy such a property, cf. Section \ref{Cohensection}.
\par\medskip
For the sake of clarity our results have been presented only for the polynomial weights $v_s$, but we remark that more general weights can also be considered, see the following Remark \ref{pesi}, $(i)$.\par

Further developments of this research could involve the study of boundedness for bilinear/multilinear pseudodifferential operators and localization operators. This requires an extension of Theorem \ref{T1} to more general  Wigner/Rihaczek distributions and STFT, see e.g., \cite{benjiKasso,COmultiloc,GT2002} and the recent contribution \cite{MOP2016}. We leave this study to a subsequent paper.

In short, the paper is organized as follows.  Section 2 is devoted to some preliminary results from time-frequency analysis and in particular to the computation of the STFT of a generalized Gaussian and its modulation norm. In Section 3 we prove Theorem \ref{T1}. In Section 4 we show the continuity properties of pseudodifferential (and in particular localization) operators on modulation spaces. In Section 5 we present a time-frequency analysis of the Cohen class. 
\par\medskip

\textbf{Notation.} We define $t^2=t\cdot t$, for $t\in\Ren$, and
$xy=x\cdot y$ is the scalar product on $\Ren$. The Schwartz class is denoted by  $\sch(\Ren)$, the space of tempered
distributions by  $\sch'(\Ren)$.   We use the brackets  $\la
f,g\ra$ to denote the extension to $\sch' (\Ren)\times\sch (\Ren)$ of
the inner product $\la f,g\ra=\int f(t){\overline {g(t)}}dt$ on
$L^2(\Ren)$. The Fourier transform is normalized to be ${\hat
	{f}}(\o)=\Fur f(\o)=\int
f(t)e^{-2\pi i t\o}dt$, the involution $g^*$ is $g^*(t) =
\overline{g(-t) }$. The operators of translation and modulation are defined by $T_x f(t)=f(t-x)$ and $M_\o f(t)=e^{2\pi i t\o} f(t)$, $x,\o\in\rd$.

\section{Preliminaries}

\subsection{Modulation spaces}
The modulation and Wiener amalgam space norms are a  measure
of
the joint time-frequency distribution of $f\in \sch '$. For their
basic properties we refer to the original literature \cite{feichtinger80,feichtinger83,feichtinger90} and the textbooks \cite{Birkbis,book}.


For the  description of the  decay  properties of a function/distribution, 
 weight functions  on the \tf\ plane are employed. We denote by  $v$  a
continuous, positive,  even, submultiplicative  weight function (in short, a
submultiplicative weight), i.e., $v(0)=1$, $v(z) = v(-z)$, and
$ v(z_1+z_2)\leq v(z_1)v(z_2)$, for all $z, z_1,z_2\in\Renn.$
A positive, even weight function $m$ on $\Renn$ is called  {\it
  v-moderate} if
$ m(z_1+z_2)\leq Cv(z_1)m(z_2)$  for all $z_1,z_2\in\Renn.$
Observe that $v_s$ is a $v_{|s|}$-moderate weight, for every $s\in\bR$.
Given a non-zero window $g\in\sch(\Ren)$, a $v$-moderate weight
function $m$ on $\Renn$, $1\leq p,q\leq
\infty$, the {\it
  modulation space} $M^{p,q}_m(\Ren)$ consists of all tempered
distributions $f\in\sch'(\Ren)$ such that the STFT $V_gf$ (defined in \eqref{eqi2}) is in $L^{p,q}_m(\Renn )$
(weighted mixed-norm spaces), with norm 
$$
\|f\|_{M^{p,q}_m}=\|V_gf\|_{L^{p,q}_m}=\left(\int_{\Ren}
  \left(\int_{\Ren}|V_gf(x,\o)|^pm(x,\o)^p\,
    dx\right)^{q/p}d\o\right)^{1/q}.  \,
$$
(Obvious modifications occur  when $p=\infty$ or $q=\infty$). If $p=q$, we write $M^p_m$ instead of $M^{p,p}_m$, and if $m(z)\equiv 1$ on $\Renn$, then we write $M^{p,q}$ and $M^p$ for $M^{p,q}_m$ and $M^{p,p}_m$.
Then  $\Mmpq (\Ren )$ is a Banach space
whose definition is independent of the choice of the window $g$, in the sense that different  nonzero window functions yield equivalent  norms.
The modulation space $M^{\infty,1}$ is also called Sj\"ostrand's class \cite{Sjostrand1}.


%
\par
We now recall the definition of the Wiener amalgam spaces that are image of the modulation spaces under the \ft.
For any  \emph{even} weigh functions $u,w$ on $\rd$, the Wiener amalgam spaces $W(\Fur L^p_u,L^q_w)(\rd)$ are given by the distributions $f\in\cS'(\rd)$ such that
\[
\|f\|_{W(\Fur L^p_u,L^q_w)(\rd)}:=\left(\int_{\Ren}
  \left(\int_{\Ren}|V_gf(x,\o)|^p u^p(\o)\,
    d\o\right)^{q/p} w^q(x)d x\right)^{1/q}<\infty  \,
\]
(with natural changes for $p=\infty$ or $q=\infty$).
Using Parseval identity we can write the so-called fundamental identity of \tfa\, $V_g f(x,\o)= e^{-2\pi i x\o}V_{\hat g} \hat f(\o,-x)$, hence $|V_g f(x,\o)|=|V_{\hat g} \hat f(\o,-x)|$  so that (recall $u(x)=u(-x)$)
\begin{equation}\label{MW}
\| f \|_{{M}^{p,q}_{u\otimes w}} = \| \hat f \|_{W(\cF L_u^p,L_w^q)}.
\end{equation}
This proves that these Wiener amalgam spaces are  the image under \ft\, of modulation spaces:
\begin{equation}\label{W-M}
\cF ({M}^{p,q}_{u\otimes w})=W(\cF L_u^p,L_w^q).
\end{equation}

In the sequel we will need the inclusion relations for modulation spaces. Assume $m_1, m_2 \in \mathcal{M}_{v}(\rdd)$, then
\begin{equation}\label{modspaceincl1}
\begin{aligned}
& \cS (\rd) \subseteq M_{m_1}^{p_1,q_1} (\rd) \subseteq M_{m_2}^{p_2,q_2} (\rd) \subseteq
\sch'(\rd), \\
& {\rm if}\qquad  p_1 \leq p_2, \quad q_1 \leq q_2, \quad m_2 \lesssim
m_1.
\end{aligned}
\end{equation}
Moreover, we will often apply convolution relations for modulation spaces \cite[Proposition 2.1]{CG02} for the $v_s$ weight functions as follows.
\begin{proposition}\label{mconvmp}
	Let $\nu (\o )>0$ be  an arbitrary  weight function on $\Ren$, $s\in\bR$,   and  $1\leq
	p,q,u,v,t\leq\infty$. If
	$$\frac1p+\frac1q-1=\frac1u,\quad \,\, \text{ and } \,
	\quad\frac1t+\frac1{t'}=\frac 1v\, ,$$
	then
	\begin{equation}\label{mconvm}
	M^{p,t}_{1\otimes \nu}(\Ren)\ast  M^{q,t'}_{1\otimes
		v_{|s|}\nu^{-1}}(\Ren)\hookrightarrow M^{u,v}_{v_s}(\Ren)
	\end{equation}
	with  norm inequality  $\| f\ast h \|_{M^{u,v}_{v_s}}\lesssim
	\|f\|_{M^{p,t}_{1\otimes \nu}}\|h\|_{ M^{q,t'}_{1\otimes
			v_{|s|}\nu^{-1}}}$.
\end{proposition}

\subsection{Time-frequency tools} To prove our main result, we will need to compute the STFT of the cross-Wigner distribution, proved in \cite[Lemma 14.5.1]{book}: 
\begin{lemma}\label{STFTSTFT}
Fix a nonzero  $g \in \cS (\Ren ) $ and let  $\Phi=W (g , g ) \in\sch(\Renn)$. Then the STFT of $W(f _1,
f _2) $ with respect to the window $\Phi $ is given by
\begin{equation}
\label{eql4}
{ {V}}_\Phi (W(f_1,f_2)) (z, \zeta ) =e^{-2\pi i
	z_2\z_2}{\overline{V_{g
		}f_2(z_1+\frac{\z_2}2,z_2-\frac{\z_1}2})}V_{g
}f_1(z_1-\frac{\z_2}2,z_2+\frac{\z_1}2)\, .
\end{equation}	
\end{lemma}

\begin{lemma}\label{1tesi}
Consider the Gaussian function $\varphi(x)=e^{-\pi x^2}$ and its rescaled version $\varphi_\lambda(x)=e^{-\pi\lambda x^2}$, $\lambda>0$. Then the cross-Wigner distribution  is the following Gaussian function
\begin{equation}\label{Wgausslambda}
 W(\varphi,\varphi_\lambda)(x,\xi)=\frac{2^d}{(1+\lambda)^{\frac d2}} e^{-a_\lambda\pi x^2}e^{-b_\lambda\pi \xi^2}e^{2\pi ic_\lambda x\xi}
 \end{equation}
with \begin{equation}\label{ee3}
a_\lambda=\frac{4\lambda}{1+\lambda}\quad b_\lambda=\frac{4}{1+\lambda}\quad  c_\lambda=\frac{2(1-\lambda)}{1+\lambda}.
\end{equation}
\end{lemma}
\begin{proof} The proof is obtained by an easy computation. In particular,  we will make the change of variables 
  $t=\frac{2s}{\sqrt{1+\lambda}}-2\frac{1-\lambda}{\sqrt{1+\lambda}}x$, so that $dt=\frac{2^d}{(1+\lambda)^{d/2}}ds$. In details,
\begin{align*}
W(\varphi,\varphi_\lambda)(x,\xi)&=\int_{\bR^{d}}e^{-\pi\big(x+\frac{t}{2}\big)^2-\pi\lambda\big(x-\frac{t}{2}\big)^2}e^{-2\pi it\xi}\;dt\\
&= e^{-\pi(1+\lambda)x^2}\int_{\bR^{d}} e^{-\frac{\pi}{4}\big[(1+\lambda)t^2+4(1-\lambda)xt\big]}e^{-2\pi it\xi}\;dt\\
&= e^{-\pi(1+\lambda)x^2}\int_{\bR^{d}} e^{-\frac{\pi}{4}\big[\sqrt{1+\lambda}t+2\frac{(1-\lambda)}{\sqrt{1+\lambda}}x\big]^2}e^{\pi\frac{(1-\lambda)^2}{1+\lambda}x^2}e^{-2\pi it\xi}\;dt\\
&= e^{-\pi\big[(1+\lambda)-\frac{(1-\lambda)^2}{1+\lambda}\big]x^2}\int_{\bR^{d}} e^{-\frac{\pi}{4}\big[\sqrt{1+\lambda}t+2\frac{1-\lambda}{\sqrt{1+\lambda}}x\big]^2}e^{-2\pi it\xi}\;dt\\
&= e^{-\pi\frac{4\lambda}{1+\lambda}x^2}\int_{\bR^d}e^{-\pi s^2}e^{-2\pi i\big(\frac{2s}{\sqrt{1+\lambda}}-\frac{2(1-\lambda)}{1+\lambda}x\big)\xi}\frac{2^d}{(1+\lambda)^{d/2}}\;ds\\
&=\frac{2^d}{(1+\lambda)^{d/2}}e^{-\pi\frac{4\lambda}{1+\lambda}x^2}e^{4\pi i\frac{1-\lambda}{1+\lambda}x\xi}\int_{\bR^d}e^{-\pi s^2}e^{-2\pi is\big(\frac{2\xi}{\sqrt{1+\lambda}}\big)}\;ds\\
&=\frac{2^d}{(1+\lambda)^{d/2}}e^{-\pi\frac{4\lambda}{1+\lambda}x^2}e^{-\pi\frac{4}{1+\lambda}\xi^2}e^{4\pi i\frac{1-\lambda}{1+\lambda}x\xi},
\end{align*}
as desired.
\end{proof}

Hence the Wigner distribution above is a generalized Gaussian. Our goal will be to compute the modulation norm of this Wigner distribution. The first step is the calculation of the STFT of a generalized Gaussian.
\begin{proposition}\label{2tesi}
Given $a,b,c>0$, consider the generalized Gaussian function  \begin{equation}\label{GG}
f(x,\xi)=e^{-\pi ax^2}e^{-\pi b\xi^2}e^{2\pi icx\xi},\quad \phas\in\rdd.
\end{equation} For $\Phi(x,\xi)=e^{-\pi(x^2+\xi^2)}$,  $z=(z_1,z_2)$, $\zeta=(\zeta_1,\zeta_2)\in\rdd$, we obtain
\begin{align}
V_\Phi f(z,\zeta)&=\frac{1}{[(a+1)(b+1)+c^2]^{\frac d2}} 
e^{-\pi\frac{[a(b+1)+c^2]z_1^2+[(a+1)b+c^2]z_2^2+(b+1)\zeta_1^2+(a+1)\zeta^2_2-2c(z_1\zeta_2+z_2\zeta_1)}{(a+1)(b+1)+c^2}} \nonumber\\
&\quad\quad\times\quad  e^{-\frac{2\pi i}{a+1} \big[z_1\zeta_1+(cz_1-(a+1)\zeta_2)\frac{c\zeta_1+(a+1)z_2}{(a+1)(b+1)+c^2}\big]}.
\label{STFTgenGauss}
\end{align}
\end{proposition}
\begin{proof} We write
\begin{align*}
V_\Phi f(z,\zeta)&=\int_{\bR^{2d}}e^{-\pi ax^2-\pi b\xi^2+2\pi icx\xi}e^{-2\pi i(\zeta_1x+\zeta_2\xi)}e^{-\pi[(x-z_1)^2+(\xi-z_2)^2]}\;dxd\xi\\
&= e^{-\pi(z_1^2+z_2^2)}\int_{\bR^d}\bigg(\int_{\bR^d}e^{-\pi[(a+1)x^2-2xz_1]}e^{-2\pi ix(\zeta_1-c\xi)}\;dx\bigg)\\&\quad\times\quad  e^{-\pi[(b+1)\xi^2-2\xi z_2]}e^{-2\pi i\zeta_2\xi}\;d\xi\\
&=e^{-\pi\big(1-\frac{1}{a+1}\big)z_1^2-\pi z_2^2}\int_{\bR^d}\bigg(\int_{\bR^d}e^{-\pi\big[\sqrt{a+1}x-\frac{z_1}{\sqrt{a+1}}\big]^2}e^{-2\pi ix(\zeta_1-c\xi)}\;dx\bigg)\\
&\quad \times\quad e^{-\pi[(b+1)\xi^2-2\xi z_2]}e^{-2\pi i\zeta_2\xi}\;d\xi .
\end{align*}
With the change of variables $ \sqrt{a+1}x-\frac{z_1}{\sqrt{a+1}}=t$, $dx=\frac{dt}{(a+1)^{d/2}}$, we obtain
\begin{align}
V_\Phi f(z,\zeta)&=\frac{1}{(a+1)^{d/2}}e^{-\pi\frac{a}{a+1}z_1^2-\pi z_2^2}\int_{\bR^d}\bigg(\int_{\bR^d}e^{-\pi t^2}e^{-2\pi i\big(\frac{t}{\sqrt{a+1}}+\frac{z_1}{a+1}\big)(\zeta_1-c\xi)}\;dt\bigg)\nonumber\\
&\quad \times\quad e^{-\pi[(b+1)\xi^2-2\xi z_2]}e^{-2\pi i\zeta_2\xi}\;d\xi \nonumber \\
&=\frac{1}{(a+1)^{d/2}}e^{-\pi\frac{a}{a+1}z_1^2-\pi z_2^2-2\pi i\frac{z_1\zeta_1}{a+1}}\int_{\bR^d}e^{-\pi\frac{(\zeta_1-c\xi)^2}{a+1}}e^{2\pi i\frac{cz_1}{a+1}\xi}\nonumber\\
&\quad \times\quad  e^{-\pi[(b+1)\xi^2-2\xi z_2]}e^{-2\pi i\zeta_2\xi}\;d\xi \nonumber\\
&=\frac{1}{(a+1)^{d/2}}e^{-\pi\frac{a}{a+1}z_1^2-\pi z_2^2-2\pi i\frac{z_1\zeta_1}{a+1}-\pi\frac{\zeta_1^2}{a+1}}\nonumber\\
&\quad \times\quad  \int_{\bR^d}e^{-\pi\big[(b+1)\xi^2+\frac{c^2}{a+1}\xi^2-2\frac{c\zeta_1}{a+1}\xi-2\xi z_2\big]} e^{2\pi i\big(\frac{cz_1}{a+1}-\zeta_2\big)\xi}\;d\xi. \label{ee1}
\end{align}
The last integral can be computed as follows:
\begin{align*}
I&:=\int_{\bR^d}e^{-\pi\big[\frac{(a+1)(b+1)+c^2}{a+1}\xi^2-2\frac{c\zeta_1+z_2(a+1)}{a+1}\xi\big]}e^{2\pi i\big(\frac{cz_1}{a+1}-\zeta_2\big)\xi}\;d\xi\\
&=\int_{\bR^d}e^{-\frac{\pi}{a+1}\{[(a+1)(b+1)+c^2]\xi^2-2[c\zeta_1+z_2(a+1)]\xi\}}e^{2\pi i\big(\frac{cz_1}{a+1}-\zeta_2\big)\xi}\;d\xi\\
&=\int_{\bR^d}e^{-\frac{\pi}{a+1}\Bigl[\sqrt{(a+1)(b+1)+c^2}\xi-\frac{c\zeta_1+z_2(a+1)}{\sqrt{(a+1)(b+1)+c^2}}\Bigl]^2}e^{\frac{\pi}{a+1}\frac{[c\zeta_1+z_2(a+1)]^2}{(a+1)(b+1)+c^2}}e^{2\pi i\big(\frac{cz_1}{a+1}-\zeta_2\big)\xi}\;d\xi\\
&=e^{\frac{\pi}{a+1}\frac{[c^2\zeta_1^2+(a+1)^2]}{(a+1)(b+1)+c^2}}\int_{\bR^d}e^{-\frac{\pi}{a+1}\Bigl[\sqrt{(a+1)(b+1)+c^2}\xi-\frac{c\zeta_1+z_2(a+1)}{\sqrt{(a+1)(b+1)+c^2}}\Bigl]^2}e^{2\pi i\big(\frac{cz_1}{a+1}-\zeta_2\big)\xi}\;d\xi.
\end{align*}
Making the change of variables
$t=\frac{\sqrt{(a+1)(b+1)+c^2}\xi-\frac{c\zeta_1+z_2(a+1)}{\sqrt{(a+1)(b+1)+c^2}}}{\sqrt{a+1}}$, so that 
$d\xi=\bigl[\frac{a+1}{(a+1)(b+1)+c^2}\bigl]^{d/2}dt$, we can write
\begin{align*}
I&=\frac{(a+1)^\frac d2}{[(a+1)(b+1)+c^2]^{\frac d2}} \int_{\bR^d}e^{-\pi t^2}e^{2\pi i \big(\frac{cz_1}{a+1}-\zeta_2\big) \Bigl[\frac{\sqrt{a+1}}{\sqrt{(a+1)(b+1)+c^2}}t-\frac{c\zeta_1+(a+1)z_2}{(a+1)(b+1)+c^2}\Bigl]}\;dt\\
&=\frac{(a+1)^\frac d2}{[(a+1)(b+1)+c^2]^{\frac d2}}e^{\frac{\pi}{a+1}\frac{(a+1)^2z_2^2+c^2\zeta_1^2+2c(a+1)z_2\zeta_1}{(a+1)(b+1)+c^2}} e^{-2\pi i \big(\frac{cz_1}{a+1}-\zeta_2\big)\frac{c\zeta_1+(a+1)z_2}{(a+1)(b+1)+c^2}} \\ &\quad \times\quad e^{-\pi\Bigl[\frac{a+1}{(a+1)(b+1)+c^2}\big(\frac{cz_1}{a+1}-\zeta_2\big)^2\Bigl]}.
\end{align*}
The result then follows by substituting the value of the integral $I$ in \eqref{ee1}.
\end{proof}
\begin{corollary}
Consider the generalized Gaussian $f$ defined in \eqref{GG} and the window function $\Phi\phas= e^{-\pi (x^2+\xi^2)}$. Then, for every $1\leq p,q\leq \infty$,  we have
\begin{equation}\label{MODGG}
\|f\|_{M^{p,q}}\asymp \|V_\Phi f\|_{L^{p,q}} \asymp \frac{[(a+1)(b+1)+c^2]^{\frac{d}{p}+\frac{d}{q}-\frac{d}{2}}[(c^2+ab+a)(c^2+ab+b)]^{\frac{d}{2q}-\frac{d}{2p}}}{[b^2(a+1)+b(c^2+a+1)]^\frac{d}{2q}[a^2(b+1)+a(c^2+b+1)]^\frac{d}{2q}}.
\end{equation}
The cases $p=\infty$ or $q=\infty$ can be obtained by using the rule $1/\infty=0$ in  formula \eqref{MODGG}. 
\end{corollary}
\begin{proof}
By Proposition \ref{2tesi}, we can write
$$|V_\Phi f(z,\zeta)|=\frac{1}{[(a+1)(b+1)+c^2]^{\frac d2}} 
e^{-\pi\frac{[a(b+1)+c^2]z_1^2+[(a+1)b+c^2]z_2^2+(b+1)\zeta_1^2+(a+1)\zeta^2_2-2c(z_1\zeta_2+z_2\zeta_1)}{(a+1)(b+1)+c^2}}.
$$
It remains to compute the mixed $L^{p,q}$-norm of the previous function. We treat the cases $1\leq p,q <\infty$. The cases either $p=\infty$  or $q=\infty$ are obtained with obvious modifications. 

For simplicity, we set 
\begin{align}\label{param1}\alpha&=\frac{c^2+a(b+1)}{(a+1)(b+1)+c^2},\quad\beta=\frac{c^2+(a+1)b}{(a+1)(b+1)+c^2},\quad \gamma=\frac{(b+1)}{(a+1)(b+1)+c^2},\\
\delta&=\frac{(a+1)}{(a+1)(b+1)+c^2}, \quad \sigma=\frac{c}{(a+1)(b+1)+c^2}.\label{param2}
\end{align}
Hence
\begin{equation}\label{ee2}
\frac{\|V_\Phi f\|_{L^{p,q}}}{[(a+1)(b+1)+c^2]^{-\frac{d}{2p}}}= \bigg(\int_{\bR^{2d}}I^{\frac{q}{p}}e^{-\pi q(\gamma \zeta_1^2+\delta \zeta_2^2)}\;d\zeta_1d\zeta_2\bigg)^{\frac{1}{q}}=:A,
\end{equation}
where
$I:=\int_{\bR^{2d}}e^{-\pi p\alpha z_1^2-\pi p\beta z_2^2}e^{2\pi p\sigma(z_1\zeta_2+z_2\zeta_1)}\;dz_1dz_2$.
Now straightforward computations and change of variables yield
\begin{align*}
 I &= \int_{\bR^{d}}\bigg(\int_{\bR^{d}}e^{-\pi p(\alpha z_1^2-2\sigma \zeta_2z_1)}\;dz_1\bigg)e^{-\pi \beta z_2^2+2\pi p\sigma z_2\zeta_1}\;dz_2\\
     &= e^{\pi p\frac{\sigma^2}{\alpha}\zeta_2^2}e^{\pi p\frac{\sigma^2}{\beta}\zeta_1^2}\int_{\bR^{d}}e^{-\pi p(\sqrt{\alpha}z_1-\frac{\sigma}{\sqrt{\alpha}}\zeta_2)^2}\;dz_1
     \int_{\bR^{d}}e^{-\pi p(\sqrt{\beta}z_2-\frac{\sigma}{\sqrt{\beta}}\zeta_1)^2}\;dz_2\\
     &= p^{-\frac{d}{2}}\alpha^{-\frac{d}{2}} p^{-\frac{d}{2}}\beta^{-\frac{d}{2}} 
     e^{\pi p\frac{\sigma^2}{\alpha}\zeta_2^2} 
     e^{\pi p\frac{\sigma^2}{\beta}\zeta_1^2}.
\end{align*}
Substituting the value of the integral $I$ in \eqref{ee2}, we obtain
\begin{align*}
A&=p^{-\frac{d}{p}}\alpha^{-\frac{d}{2p}}\beta^{-\frac{d}{2p}}
    \bigg(\int_{\bR^{d}}e^{\pi q\frac{\sigma^2}{\beta}\zeta_1^2-\pi q\gamma \zeta_1^2}\;d\zeta_1\int_{\bR^{d}} e^{\pi q\frac{\sigma^2}{\alpha}\zeta_2^2 -\pi q\delta \zeta_2^2}\;d\zeta_2\bigg)^\frac{1}{q}\\
    &=p^{-\frac{d}{p}}\alpha^{-\frac{d}{2p}}\beta^{-\frac{d}{2p}}
    \bigg(\int_{\bR^{d}}e^{-\pi q(\gamma-\frac{\sigma^2}{\beta}) \zeta_1^2}\;d\zeta_1  
    \int_{\bR^{d}} e^{-\pi q(\delta-\frac{\sigma^2}{\alpha})\zeta_2^2} \bigg)^\frac{1}{q}\\
     &=p^{-\frac{d}{p}}\alpha^{-\frac{d}{2p}}\beta^{-\frac{d}{2p}}
     q^{-\frac{d}{q}}\Big(\gamma-\frac{\sigma^2}{\beta}\Big)^{-\frac{d}{2q}}\Big(\gamma-\frac{\sigma^2}{\alpha}\Big)^{-\frac{d}{2q}}. 
\end{align*}
Finally, the goal is attained by substituting in $A$ the values of the parameters $\alpha,\beta,\gamma,\delta,\sigma$ in \eqref{param1} and \eqref{param2} and observing that 
$$\|f\|_{M^{p,q}}\asymp\|V_\Phi f\|_{L^{p,q}}=A [(a+1)(b+1)+c^2]^{-\frac d2}.$$
This concludes the proof.
\end{proof}

We have now all the tools to compute the modulation norm of the (cross-)Wigner distribution $W(\varphi,\varphi_\lambda)$ in \eqref{Wgausslambda}. Precisely, setting   in formula  \eqref{MODGG} the values $a=a_\lambda$, $b=b_\lambda$, $c=c_\lambda$,  where $a_\lambda, b_\lambda, c_\lambda$ are defined in \eqref{ee3}, and making easy simplifications we attain the following result.
\begin{corollary}\label{ee7}
For $\lambda>0$ consider the (cross-)Wigner distribution $W(\varphi,\varphi_\lambda)$ defined in \eqref{Wgausslambda} (cf. Lemma \ref{1tesi}). Then  
\begin{equation}\label{ee4}
\|W(\varphi,\varphi_\lambda)\|_{M^{p,q}}\asymp\frac{[(2\lambda+1)(\lambda+2)]^{\frac{d}{2q}-\frac{d}{2p}}}{\lambda^\frac{d}{2q}(1+\lambda)^{\frac{d}{2}-\frac dp}}.
\end{equation}
The cases $p=\infty$ or $q=\infty$ can be obtained by using the rule $1/\infty=0$ in  formula \eqref{ee4}.
\end{corollary}

\section{Main Result}


In this Section we prove Theorem \ref{T1}. We will focus separately on the sufficient and necessary part in the statement.  


\begin{theorem}[\bf Sufficient Conditions]\label{wigestp}
	If $ p_1,q_1,p_2,q_2,p,q\in [1,\infty]$ are indices which satisfy \eqref{WIR} and \eqref{Wigindexsharp}, $s\in\bR$, $f_1\in M^{p_1,q_1}_{v_{|s|}}(\Ren)$ and
	$f_2\in M^{p_2,q_2}_{v_s}(\Ren)$, then  $W(f_1,f_2)\in
	M^{p,q}_{1\otimes v_s}(\Renn)$, and the estimate \eqref{wigest} holds true.
\end{theorem}
\begin{proof}  
We first study the case $p,q<\infty$. Let $g\in \cS (\rd ) $ and set $\Phi=W(g,g)\in\sch(\Renn)$.  If $\zeta
	= (\z_1,\z_2)\in
	\Renn$, we write $\tilde{\zeta } = (\zeta _2,-\zeta _1)$. Then, from 
	Lemma \ref{STFTSTFT}, 
	\begin{equation}\label{e0}
	|{{V}}_\Phi (W(f_1,f_2))(z,\zeta)| =| V_g f_2(z
	+\tfrac{\tilde{\z }}{2})| \,  |V_g f_1(z - \tfrac{\tilde{\z }}{2})| \,.
	\end{equation}
	Consequently,
	\begin{equation*}
	\|W( f_1,f_2)\|_{M^{p,q}_{1\otimes v_s}}  \asymp
	\left(\intrdd\!\left(\intrdd\!
	| V_g f_2(z +\tfrac{\tilde{\z }}{2})|^p \,  |V_g f_1(z -
	\tfrac{\tilde{\z }}{2})|^p   \, dz \right)^\frac{q}{p} \, \langle \zeta \rangle
	^{sq} \, d\zeta \right)^{1/q}.
	\end{equation*}
	After  the change of variables $z \mapsto z-\tilde{\zeta } /2$, the
	integral over $z$ becomes the convolution $(|V_g f_2|^p\ast
	|(V_g{f_1})^*|^p)(\tilde{\zeta })$,
	and observing that $(1\otimes v_s) (z,\zeta ) = \langle \zeta \rangle ^s =
	v_s (\zeta )= v_s (\tilde{\zeta })$, we obtain
	\begin{eqnarray*}
		\|W(f_1,f_2)\|_{M^{p,q}_{1\otimes v_s}}
		&\asymp&
		\left(\iint_{\Renn}\!(|V_g f_2|^p\ast |(V_g {f_1})^*|^p)^\frac{q}{p}(\tilde{\zeta})
		v_s(\tilde{\zeta})^{q} \, d\zeta \right)^{1/p}\\
		&=& \| \, |V_g f_2|^p\ast |(V_g{ f_1})^*|^p \, \|^{\frac 1 p}_{L^\frac{q}{p}_{v _{ps}}}.
	\end{eqnarray*}
	Hence 
	\begin{equation}\label{e1}
		\|W(f_1,f_2)\|^p_{M^{p,q}_{1\otimes v_s}}\asymp \| \, |V_g f_2|^p\ast |(V_g{ f_1})^*|^p \, \|_{L^\frac{q}{p}_{v _{ps}}}.
	\end{equation}
\emph{Case $p\leq q<\infty$}.\par {\bf Step 1.}  Here we prove the desired result in the case $p\leq p_i,q_i$, $i=1,2$.\par
Suppose first that \eqref{Wigindex} are satisfied (and hence $p_i,q_i\leq q$,  $i=1,2$).
Since $q/p\geq 1$, we can apply Young's Inequality  for mixed-normed spaces	(cf. \cite{BP61}, see also \cite{Galperin2014}) and majorize \eqref{e1} as follows
	\begin{align*}
		\|W(f_1,f_2)\|^p_{M^{p,q}_{1\otimes v_s}}&\lesssim 
		\| \, |V_g f_2|^p\|_{L^{r_2,s_2}_{v _{p|s|}}}\|\, |(V_g{ f_1})^*|^p  \|_{L^{r_1,s_1}_{v _{ps}}} \, \\
		&= \| |V_g{ f_1}|^p  \|_{L^{r_1,s_1}_{v _{p|s|}}} \| \, |V_g f_2|^p\|_{L^{r_2,s_2}_{v _{ps}}}\, \\
		&= \| V_g{ f_1}  \|^p_{L^{pr_1,ps_1}_{v _{|s|}}} \| V_g f_2\|^p_{L^{p r_2,ps_2}_{v _{s}}}\, ,
	\end{align*}
	for every $1\leq r_1,r_2,s_1,s_2\leq\infty$ such that 
	 \begin{equation}\label{e2}
	\frac 1{r_1}+\frac{1}{r_2}=\frac 1{s_1}+\frac{1}{s_2}=1+ \frac{p}{q}.
	\end{equation}
Choosing $r_i=p_i/p\geq 1$, $s_i=q_i/p\geq 1$, $i=1,2$, the indices' relation \eqref{e2} becomes \eqref{Wigindex}  and we obtain 
$$	\|W(f_1,f_2)\|_{M^{p,q}_{1\otimes v_s}}\lesssim \| V_g{ f_1}  \|_{L^{p_1,q_1}_{v _{|s|}}} \| V_g f_2\|_{L^{p_2,q_2}_{v _{s}}}\asymp \|f_1\|_{M^{p_1,q_1}_{v_{|s|}}}\|f_2\|_{M^{p_2,q_2}_{v_s}}.
$$
Now, still assume $p\leq p_i,q_i$, $i=1,2$ but
$$\frac 1{p_1}+\frac 1{p_2}\geq \frac{1}{p}+\frac 1q,\quad \frac 1{q_1}+\frac 1{q_2}= \frac{1}{p}+\frac 1q,
$$
(hence $p_i,q_i\leq q$, $i=1,2$). We set $u_1=t p_1$,  and look for $t\geq 1$ (hence $u_1\geq p_1$) such that
$$\frac 1 {u_1}+\frac 1{p_2}= \frac{1}{p}+\frac 1q
$$
that gives
$$0<\frac 1t=\frac{p_1}{p}+\frac{p_1}{q}-\frac{p_1}{p_2}\leq1
$$
because $p_1(1/p+1/q)-p_1/p_2\leq p_1(1/p_1+1/p_2)-p_1/p_2=1$ whereas the lower bound of the previous estimate follows by   $1/(tp_1)=1/p+1/q-1/p_2>0$ since $p\leq p_2$.
Hence the previous part of the proof gives 
\begin{align*}
	\|W(f_1,f_2)\|_{M^{p,q}_{1\otimes v_s}}&\lesssim \|f_1\|_{M^{u_1,q_1}_{v_{|s|}}}\|f_2\|_{M^{p_2,q_2}_{v_s}}\\
	&\lesssim \|f_1\|_{M^{p_1,q_1}_{v_{|s|}}}\|f_2\|_{M^{p_2,q_2}_{v_s}}.	
\end{align*}
where the last inequality follows by inclusion relations for modulations spaces $M^{p_1,q_1}_{v_s}(\rd)\subseteq M^{u_1,q_1}_{v_s}(\rd)$ for $p_1\leq u_1$.

The general case 
$$\frac 1{p_1}+\frac 1{p_2}\geq  \frac{1}{p}+\frac 1q,\quad \frac 1{q_1}+\frac 1{q_2}\geq \frac{1}{p}+\frac 1q,
$$
 can be treated analogously.\par
\noindent {\bf Step 2}. Assume now that $p_i,q_i\leq q$, $i=1,2$, and satisfy relation \eqref{Wigindexsharp}.  If at least one out of the indices $p_1, p_2$ is less than $p$, assume for instance $p_1\leq p$, whereas $p\leq q_1,q_2$, then we proceed as follows. 
We choose $u_1=p$, $u_2=q$, and deduce by the results in Step 1 (with $p_1=u_1$ and $p_2=u_2$) that
$$ 	\|W(f_1,f_2)\|_{M^{p,q}_{1\otimes v_s}}\lesssim \|f_1\|_{M^{u_1,q_1}_{v_{|s|}}}\|f_2\|_{M^{u_2,q_2}_{v_s}}\lesssim\|f_1\|_{M^{p_1,q_1}_{v_{|s|}}}\|f_2\|_{M^{p_2,q_2}_{v_s}}
$$ 
where the last inequality follows by inclusion relations for modulation spaces, since 
$p_1\leq u_1=p$ and $p_2\leq u_2=q$.\par
Similarly we argue when at least one out of the indices $q_1, q_2$ is less than $p$ and $p\leq p_1,p_2$ or when at least one out of the indices $q_1, q_2$ is less than $p$ and at least one out of the indices $q_1, q_2$ is less than $p$. The remaining case $p\leq p_i,q_i\leq q$ is treated in Step 1.

\noindent\emph{Case $p<q=\infty$.} The argument are similar to the case $p\leq q<\infty$.

\noindent\emph{Case $p=q=\infty$.}  We use \eqref{e0} and the submultiplicative property of the weight $v_s$,
\begin{align*}
\|W(f_1,f_2)\|_{M^{\infty}_{1\otimes v_s}}&=\sup_{z,\zeta\in\rdd}| V_g f_2(z
+\tfrac{\tilde{\z }}{2})| \,  |V_g f_1(z - \tfrac{\tilde{\z }}{2})| v_s(\zeta)\\
&=\sup_{z,\zeta\in\rdd}|| V_g f_2(z)| \,  |(V_g f_1)^*(z - \tilde{\z })| v_s(\z )\\
&=\sup_{z,\zeta\in\rdd}|| V_g f_2(z)| \,  |(V_g f_1)^*(z - \tilde{\z })| v_s(\tilde{\z })\\
&\leq \sup_{z\in\rdd}(\|V_g f_1 v_{|s|}\|_{\infty} \,|V_g f_2(z) v_s(z)|)= \|V_g f_1 v_{|s|}\|_{\infty}\|V_g f_2 v_s\|_{\infty}\\
&\asymp\|f\|_{M^\infty_{v_{|s|}}}\|g\|_{M^\infty_{v_s}}\leq \|f\|_{M^{p_1,q_1}_{v_{|s|}}}\|f\|_{M^{p_2,q_2}_{v_s}},
\end{align*}
for every $1\leq p_i,q_i\leq \infty$, $i=1,2$. Notice that in this case conditions \eqref{WIR} and \eqref{Wigestsharp} are trivially satisfied.

\noindent\emph{Case $p>q$.} Using the inclusion relations for modulation spaces, we majorize 
$$\|W(f_1,f_2)\|_{M^{p,q}_{1\otimes v_s}}\lesssim \|W(f_1,f_2)\|_{M^{q,q}_{1\otimes v_s}}\lesssim \|f_1\|_{M^{p_1,q_1}_{v_{|s|}}}\|f_2\|_{M^{p_2,q_2}_{v_s}}$$
for every $1\leq p_i,q_i\leq q$, $i=1,2$. Here we have applied  the case $p\leq q$ with $p=q$. Notice that in this case condition \eqref{Wigestsharp} is trivially satisfied, since from $p_1,q_i\leq q$ we infer $1/p_1+1/p_2\geq 1/q+1/q$, $1/q_1+1/q_2\geq 1/q+1/q$. This concludes the proof.
\end{proof}
\begin{remark}\label{pesi} \par\indent
(i) The result of Theorem \ref{wigestp} can be extended to more general  weights. In particular, it holds for polynomial weights satisfying relation $(4.10)$ in \cite{Toftweight}.  Hence our result extends Toft's result \cite[Theorem 4.2]{Toftweight}  (cf.\ also \cite[Theorem 4.1]{toft1} for \modsp s without weights). 
Other examples of suitable weights are given by sub-exponential weights of the type $v(z)= e^{\alpha |z|^{\beta}}$ for $\alpha>0$ and $0<\beta<1$.

\indent (ii) The particular case $p=1$, $1\leq q \leq \infty$, $p_1=q_1=1$, $p_2=q_2=q$, $s\geq 0$, was already proved in \cite[Prop. 2.2]{CG02}.\par

\indent (iii) For $p_i=q_i=p=q=2$, $i=1,2$, we obtain the following continuity result for the cross-Wigner distribution acting between Shubin spaces and Sobolev spaces:

For $s\geq 0$,  $f_1,f_2\in \Qs(\Ren)$ (cf. Shubin's book \cite{Shubin91}), the cross-Wigner distribution  $W(f_1,f_2)$ is in $H^s(\rdd)$ with 
$$\|W(f_1,f_2)\|_{H^s(\rdd)}\lesssim \|f_1\|_{\Qs(\Ren)}\|f_2\|_{\Qs(\Ren)}.
$$

\indent (iv) Continuity properties of the cross-Wigner distribution on modulation spaces with different weight functions can be easily inferred using the techniques of Theorem \ref{wigestp} and  the Young type inequalities for weighted spaces shown by Johansson et al. in  \cite[Theorem 2.2]{JPTT2015}.
\end{remark}
The estimate in \eqref{wigest} can be slightly improved if $s\geq0$. Precisely, we have the following result. 
\begin{theorem}\label{wigestpbis}
	If $ p_1,q_1,p_2,q_2,p,q\in [1,\infty]$ are indices which satisfy \eqref{WIR} and \eqref{Wigindexsharp}, $s\geq 0$, $f_1\in M^{p_1,q_1}_{v_{s}}(\Ren)$ and
	$f_2\in M^{p_2,q_2}_{v_s}(\Ren)$, then  $W(f_1,f_2)\in
	M^{p,q}_{1\otimes v_s}(\Renn)$, with
	\[
	\| W(f_1,f_2)\|_{M^{p,q}_{1\otimes v_s}}\lesssim
	\|f_1\|_{M^{p_1,q_1}}\| f_2\|_{M^{p_2,q_2}_{v_s}}+\|f_1\|_{M^{p_1,q_1}_{v_s}}\| f_2\|_{M^{p_2,q_2}}.
\]
\end{theorem}
\begin{proof}  
The proof is similar to that of Theorem \ref{wigestp}, but in this case from the estimate \eqref{e1} we proceed by using
\[
v_{s}(z)\lesssim v_s(z-w)+v_s(w),\quad s\geq 0 
\]
(with $sp$ in place of $s$) instead of $ v_{s}(z)\lesssim v_{|s|}(z-w)v_s(w)$.
\end{proof}

If in particular we consider the Wigner distribution $W(f,f)$, then Theorem \ref{wigestpbis} can be rephrased as follows.
\begin{corollary}\label{Wignerdistr}
	Assume $s\geq0$, $p_1,q_1,p,q \in [1,\infty]$ such that
	\begin{equation*}
 2\min \{ \frac 1{p_1},\frac 1{q_1}\}\geq \frac 1p +\frac 1 q.
	\end{equation*}
	 If $f\in M^{p_1,q_1}_{v_s}(\rd)$, then the Wigner distribution  $W(f,f)$ is in $ M^{p,q}_{1\otimes v_{s}}(\rdd)$, with
	\begin{equation}\label{Wigner}
	\|W(f,f)\|_{M^{p,q}_{1\otimes v_{s}}(\rdd)}\lesssim \|f\|_{M^{p_1,q_1}(\rd)}\|f\|_{M^{p_1,q_1}_{v_s}(\rd)}.
	\end{equation}
\end{corollary}

We now prove the sharpness of Theorem \ref{wigestp} (and Corollary \ref{Wignerdistr}) in the un-weighted case $s=0$.
\begin{theorem}[\bf Necessary Conditions] Consider $p_1,p_2,q_1,q_2,p,q\in [1,\infty]$. Assume that there exists a constant $C>0$ such that
	\begin{equation}\label{Wigestsharp2}
	\|W(f_1,f_2)\|_{M^{p,q}}\leq C \|f_1\|_{M^{p_1,q_1}} \|f_2\|_{M^{p_2,q_2}},\quad \forall f_1,f_2\in\cS(\rdd),
	\end{equation}
	then \eqref{WIR} and \eqref{Wigindexsharp} must hold.
\end{theorem}
\begin{proof}	
Let us first demonstrate the necessity of \eqref{Wigindexsharp}.  We consider the dilated Gaussians $\f_\lambda(x)=\f(\sqrt{\lambda} x)$, with $\f(x)=e^{-\pi x^2}$.

An easy computation (see also \cite[formula (4.20)]{book} \eqref{Wigindexsharp}) shows that
\begin{equation}\label{Wiggauss}
    W(\f_\lambda,\f_\lambda)\phas=2^{\frac d2} \lambda^{-\frac d2}\f_{ 2 \lambda}(x)\f_{{2} /\lambda}(\o).
\end{equation}
Now (see \cite[Lemma 3.2]{cordero2}, \cite[Lemma 1.8]{toft1})
$$ \|\f_\lambda\|_{M^{r,s}}\asymp \lambda^{-\frac d {2r}}(\lambda+1)^{-\frac d2(1-\frac1q-\frac1p)}$$
and observe that
$$\| W(\f_\lambda,\f_\lambda)\|_{M^{p,q}(\rdd)}= 2^{\frac d2} \lambda^{-\frac d2}\|\f_{ 2 \lambda}\|_{M^{p,q}(\rd)}\|\f_{ {2} /\lambda}\|_{M^{p,q}(\rd)}.
$$
The assumption \eqref{Wigestsharp} in this case becomes
$$\lambda^{-\frac d2}(\lambda+1)^{-\frac d2(1-\frac1q-\frac1p)}(\lambda^{-1}+1)^{-\frac d2(1-\frac1q-\frac1p)}\lesssim \lambda^{-\frac d{2p_1}}(1+\lambda)^{-\frac d2(1-\frac1{q_1}-\frac1{p_1})}\lambda^{-\frac d{2p_2}}(1+\lambda)^{-\frac d2(1-\frac1{q_2}-\frac1{p_2})}
$$
and letting $\lambda\to+\infty$ we obtain
$$ \frac1p+\frac1q\leq \frac1{q_1}+\frac1{q_2}
$$
whereas for $\lambda\to 0^+$
$$ \frac1p+\frac1q\leq \frac1{p_1}+\frac1{p_2},
$$
so that \eqref{Wigindexsharp} must hold.\par

It remains to prove the sharpness of \eqref{WIR}. We first show the conditions $p_2,q_2\leq q$.
We test \eqref{Wigestsharp2} on the (cross-)Wigner distribution $W(\f,\f_\lambda)$ defined in \eqref{Wgausslambda}, that is
$$\|W(\f,\f_\lambda)\|_{M^{p,q}(\rdd)}\lesssim \|\f\|_{M^{p_1,q_1}(\rd)}\|\f_\lambda\|_{M^{p_2,q_2}(\rd)}.
$$
Using Corollary \ref{ee7} the previous estimate can be rephrased as
$$\frac{[(2\lambda+1)(\lambda+2)]^{\frac{d}{2q}-\frac{d}{2p}}}{\lambda^\frac{d}{2q}(1+\lambda)^{\frac{d}{2}-\frac dp}}\lesssim \lambda^{-\frac d{2p_2}}\lambda^{-\frac d2(1-\frac1{q_2}-\frac1{p_2})},\quad \forall\lambda>0.
$$
Letting $\lambda\to+\infty$ we attain
$$ q_2\leq q
$$
whereas for  $\lambda\to 0^+$ 
$$ p_2\leq q.$$
The conditions $p_1,q_1\leq q$ then follows by using  the cross-Wigner property $$W(\f_\lambda,\f)\phas=\overline{W(\f,\f_\lambda)\phas},$$ so that 
$$\|W(\f_\lambda,\f)\|_{M^{p,q}(\rdd)}=\|\overline{W(\f,\f_\lambda)}\|_{M^{p,q}(\rdd)}=\|W(\f,\f_\lambda)\|_{M^{p,q}(\rdd)}$$
 and applying the same argument as before. 
\end{proof}\par
\medskip
\section{Continuity results for the short-time Fourier transform and the ambiguity distribution}
This optimal bounds in Theorem \ref{T1} for the Wigner distribution can be translated in optimal new estimates for other time-frequency representations such that the STFT or the ambiguity function. Precisely, 
given $f,g\in\lrd$, we recall the definition of the (cross-)ambiguity function 
\begin{equation}\label{AF}
A(f_1,f_2)\phas= \intrd e^{-2\pi i t\xi} f_1(t+\frac x2)\overline{f_2(t-\frac x 2)}\,dt.
\end{equation}
It is well-known that the Wigner distribution is the symplectic Fourier transform of the ambiguity function, see e.g., \cite{Birkbis}. In other words, cf. \cite[Lemma 4.3.4]{book}, 
\begin{equation}\label{e3}
W(f_1,f_2)\phas=\cF \cU A(f_1,f_2)\phas,\quad f_1,f_2\in\lrd,
\end{equation}
where the operator $\cU$ is the rotation $\cU F\phas=F(\o,-x)$ of a function $F$ on $\rdd$.
We need the following norm equivalence.
\begin{lemma} For $s\in\bR$, $1\leq p,q\leq \infty$,
 the following equivalence holds
$$\|W(f_1,f_2)\|_{M^{p,q}_{1\otimes v_s}}=\|A(f_1,f_2)\|_{W(\cF L^p,L^q_{v_s})}\asymp 
 \|V_{f_2}f_1\|_{W(\cF L^p,L^q_{v_s})}.
$$
\end{lemma}\begin{proof}
Let us observe that the weight $v_s$, $s\in\bR$,  is symmetric in each coordinate:
$$v_s\phas=v_s(x,-\o)=v_s(-x,\o)=v_s(-x,-\o).
$$
 Using \eqref{e3}, the connection between modulation and Wiener amalgam spaces \eqref{MW} and the symmetry of the weights $v_s$ we can write
\begin{align*}
\|W(f_1,f_2)\|_{M^{p,q}_{1\otimes v_s}}&=\|\cF \cU A(f_1,f_2)\|_{M^{p,q}_{1\otimes v_s}}=\|\cU A(f_1,f_2)\|_{W(\cF L^p,L^q_{v_s})}\\
&=\| A(f_1,f_2)\|_{W(\cF L^p,L^q_{v_s})}.
\end{align*}
Now a simple change of variables in \eqref{AF} let us write
$$A(f_1,f_2)\phas= e^{\pi i x\xi} V_{f_2}f_1\phas.
$$
It was proved in \cite[Proposition 3.2]{CDNACHA2016} that the function $F\phas=e^{\pi i x\xi}$ is in the Wiener amalgam space $W(\cF L^1, L^\infty)$.  This means that, by the product properties for Wiener amalgam spaces, for every $s\in\bR$,
$$\|A(f_1,f_2)\|_{W(\cF L^p,L^q_{v_s})}\lesssim \|F\|_{W(\cF L^1,L^\infty)} \|V_{f_2}f_1\|_{W(\cF L^p,L^q_{v_s})}
$$
and since $\bar{F}\phas=e^{-\pi i x\xi}\in W(\cF L^1, L^\infty)$ as well, with 
$\|\bar{F}\|_{W(\cF L^1,L^\infty)}=\|F\|_{W(\cF L^1,L^\infty)}$,  we can analogously write
$$ \|V_{f_2}f_1\|_{W(\cF L^p,L^q_{v_s})}\lesssim \|F\|_{W(\cF L^1,L^\infty)} \|A(f_1,f_2)\|_{W(\cF L^p,L^q_{v_s})}.
$$ This proves the desired result.
\end{proof}

This observations, together with the Wigner property $W(f_1,f_2)\phas=\overline{W(f_2,f_1)}$ let us translate Theorem \ref{T1} in terms of STFT acting from modulation spaces to Wiener amalgam spaces. Notice that the following two corollaries also hold  for the ambiguity function $A(f_1,f_2)$ in place  of the STFT $V_{f_1}f_2$.
\begin{corollary}\label{e6}
Consider $s\in\bR$ and assume that $p_1,p_2,q_1,q_2,p,q\in [1,\infty]$ satisfy conditions \eqref{WIR} and \eqref{Wigindexsharp}. Then if $f_1\in M^{p_1,q_1}_{v_{|s|}}(\rd)$ and $f_2\in M^{p_2,q_2}_{v_s}(\rd)$, we have $V_{f_1}f_2\in W(\cF L^p,L^q_{v_s})(\rdd)$ with
\begin{equation}\label{Wigest2}
\|V_{f_1}f_2\|_{W(\cF L^p,L^q_{v_s})(\rdd)}\lesssim \|f_ 1\|_{M^{p_1,q_1}_{v_s}(\rd)} \|f_2\|_{M^{p_2,q_2}_{v_s}(\rd)}.
\end{equation}
 Viceversa, assume that there exists a constant $C>0$ such that
\begin{equation}\label{STFTsharp}
\|V_{f_1}f_2\|_{W(\cF L^p,L^q)(\rdd)}\leq C \|f_1\|_{M^{p_1,q_1}(\rd)} \|f_2\|_{M^{p_2,q_2}(\rd)},\quad \forall f_1,f_2\in\cS(\rdd).
\end{equation}
Then \eqref{WIR} and \eqref{Wigindexsharp} must hold.
\end{corollary}
\vskip0.5truecm

The previous result has many special and interesting cases. Let us just give a flavour of the main important ones. For $p_i=q_i$, $i=1,2$,  we obtain the following result.
\begin{corollary}\label{e8}
	Assume that $p_1,p_2,p,q\in [1,\infty]$ satisfy 
	\begin{equation}\label{e5}
p_1,p_2\leq q,\quad \frac1{p_1}+\frac1{p_2}\geq \frac1p+\frac1q.
	\end{equation}  Then, for $s\in\bR$,  if $f_1\in M^{p_1}_{v_{|s|}}(\rd)$ and $f_2\in M^{p_2}_{v_s}(\rd)$  we have $V_{f_1}f_2\in W(\cF L^p,L^q_{v_s})(\rdd)$ with
	\begin{equation}\label{Wigest2p}
	\|V_{f_1}f_2\|_{W(\cF L^p,L^q_{v_s})(\rdd)}\lesssim \|f_ 1\|_{M^{p_1}_{ v_{|s|}}(\rd)} \|f_2\|_{M^{p_2}_{ v_s}(\rd)}.
	\end{equation}
	\par Viceversa, assume that there exists a constant $C>0$ such that
	\begin{equation}\label{WigestsharpSTFT}
	\|V_{f_1}f_2\|_{W(\cF L^p,L^q)(\rdd)}\leq C \|f_1\|_{M^{p_1}(\rd)} \|f_2\|_{M^{p_2}(\rd)},\quad \forall f_1,f_2\in\cS(\rdd).
	\end{equation}
	Then \eqref{e5} must hold.
\end{corollary}
\begin{remark}
	The previous result holds also for the cross-Wigner distribution if we replace  $\|V_{f_1}f_2\|_{W(\cF L^p,L^q_{v_s})(\rdd)}$ by 	$\|W(f_1,f_2)\|_{M^{p,q}_{1\otimes  v_s}(\rdd)}$
	\end{remark}
If we choose $s=0$, $p=q'$ and $q\geq2 $ in the previous result, we can refine  some  Lieb's integral bounds for the ambiguity function showed in \cite{Lieb}. Namely, we obtain in particular the following sufficient conditions for boundedness.
\begin{corollary}\label{e7} Assume $q\geq 2$, $p_1,p_2,q_1,q_2\leq q$ such that 
	$$\frac1{p_1}+\frac{1}{p_2}\geq 1\quad \frac1{q_1}+\frac{1}{q_2}\geq 1.
	$$
If $f_i \in M^{p_i,q_i}(\rd)$, $i=1,2$, then the ambiguity function satisfy $A(f_1,f_2)\in L^q(\rdd)$, with
$$\|A(f_1,f_2)\|_{L^q(\rdd)}\lesssim \|f_1\|_{M^{p_1,q_1}(\rd)}\|f_2\|_{M^{p_2,q_2}(\rd)}.$$
\end{corollary}
\begin{proof}
Since $\cF L^{q'}\subseteq L^q$ for $q\geq 2$, the inclusion relations for Wiener amalgam spaces give 
$W(\cF L^{q'}, L^q)(\rdd)\subseteq W(L^q,L^q)(\rdd)=L^q(\rdd)$. The result then follows by Corollary \ref{e6}.
\end{proof}

Observe that for $p_i=q_i=2$, $i=1,2$, $M^{p_i,q_i}(\rd)=L^2(\rd)$ and we recapture 
Lieb's bound, see \cite[Theorem 1]{Lieb}. We also refer to \cite{cordero1,cordero2} for related estimates for the short-time Fourier transform.

\section{Pseudodifferential operators}\label{psdo}
In this section we apply Theorem \ref{T1} to the study of pseudodifferential operators on modulation spaces. The key tool is the weak definition of a Weyl operator $L_\sigma$ by means of a duality pairing between the symbol $\sigma$ and the cross-Wigner distribution $W(g,f)$ as shown in \eqref{equiv1}. \par
The sharpest result concerning boundedness of pseudodifferential operators on (un-weighted) modulation spaces was proved by one of us with Tabacco and Wahlberg  in \cite[Theorem 1.1]{cordero3}. Such result covers
previous sufficient boundedness conditions proved by Toft in \cite[Theorem 4.3]{toft1} and necessary boundedness conditions exhibited in our previous  work \cite[Proposition 5.3]{cordero2}.
Our result in this framework extends \cite[Theorem 1.1]{cordero3} to  weighted modulation spaces, thus widening the sufficient boundedness conditions presented by Toft in \cite[Theorem 4.3]{Toftweight}.
Using Theorem \ref{T1} the proof of the following result is decidedly simple.
\begin{theorem}\label{Charpseudo}
	Assume $s\geq0$,  $p_i,q_i,p,q\in [1,\infty]$, $i=1,2$, are such that \begin{equation}\label{indicitutti}
	\min\{\frac1{p_1} + \frac{1}{p'_2}, \frac1{q_1} +\frac{1}{q'_2}\} \geq \frac1{p'}+\frac1{q'}.
	\end{equation}
	\noindent
	and
	\begin{equation}\label{indiceq}
	\quad q \leq \min\{ p_1',q_1',p_2,q_2\}.
	\end{equation}
	Then the
	pseudodifferential operator $T$, from $\cS(\rd)$ to $\cS'(\rd)$,
	having symbol  $\sigma \in M^{p,q}_{1\otimes v_s}(\R^{2d})$, extends uniquely to
	a bounded operator from ${M}^{p_1,q_1}_{v_s}(\R^d)$ to
	${M}^{p_2,q_2}_{v_s}(\R^d)$, with the estimate
	\begin{equation}\label{stimaA}
	\|Tf\|_{{M}_{v_s}^{p_2,q_2}} \lesssim
	\|\sigma\|_{M^{p,q}_{1\otimes v_s}}\|f\|_{{M}_{v_s}^{p_1,q_1}}.
	\end{equation}	
	Vice-versa, if \eqref{stimaA} holds for $s=0$ and for every $f\in\cS(\rd)$, $\sigma\in \cS'(\rdd)$, then \eqref{indicitutti} and \eqref{indiceq} must be satisfied. 
\end{theorem}
\begin{proof} Assume $\sigma \in M^{p,q}_{1\otimes v_s}(\R^{2d})$, $f\in {M}^{p_1,q_1}_{v_s}(\R^d)$ such that \eqref{indicitutti} and \eqref{indiceq} are satisfied. For $g\in {M}^{p'_2,q'_2}_{v_{-s}}(\R^d)$ Theorem \ref{T1} says that the cross-Wigner distribution is in $M^{p',q'}_{1\otimes v_{-s}}(\R^{2d})$, provided that $p_1,q_1,p'_2,q'_2\leq q'$ and $$\min\{1/p_1+1/p'_2, 1/q_1+1/q_1'\}\geq 1/p'+1/q',$$ that are  conditions \eqref{indiceq} and \eqref{indicitutti}, respectively. Thereby there exists a constant $C>0$ such that
	\begin{align*}
	|\la \sigma, W(g,f)\ra|&\leq \|a\|_{M^{p,q}_{1\otimes v_s}(\R^{2d})}\|W(g,f)\|_{M^{p',q'}_{1\otimes v_{-s}}(\R^{2d})}\\
	&\leq C \|f\|_{{M}^{p_1,q_1}_{v_s}(\R^d)} \|g\|_{{M}^{p'_2,q'_2}_{v_{-s}}(\R^d)}.
	\end{align*}
Since $|\la L_\sigma f,g\ra|=|\la \sigma, W(g,f)\ra|$, this concludes the proof of the sufficient conditions. The necessary conditions are proved in \cite[Theorem 1.1]{cordero3}.
\end{proof}

We now present sharp boundedness results for localization operators. 

Let us mention that, since their introduction by Daubechies \cite{Daube88} as a mathematical tool to localize a signal in the time-frequency plane, they have been investigated by many authors in the field of signal analysis, see \cite{BCG2004,CG02,Fei-Now02,RT93,toft1,Wong02,Nenadmed2016} and references therein. Localization operators with Gaussian windows are well-known in quantum mechanics, under the name of anti-Wick operators \cite{Berezin71,Shubin91}. 

A localization operator  $\aaf $ with symbol $a$ and
windows $\f _1, \f _2$ is defined   as
\begin{equation}
\label{eqi4}
\aaf f(t)=\int_{\Renn}a \phas V_{\f _1}f \phas M_\o T_x \f _2 (t)
\,
dx d\o \, .
\end{equation}
In signal analysis the meaning is as follows: first, analyse the signal $f$ by taking the STFT $V_{\f_1}f $, then localize $f$ by multiplying with the symbol $a$ (if in particular $a = \chi _{\Omega }$, for some compact set $\Omega \subseteq \rdd
$,   it is considered only the part of $f$
that lives on the set $\Omega $ in the \tf\ plane), then reconstruct the signal by superposition of \tfs s  with respect to the window $\f_2$.
 If $\f _1(t)  = \f _2 (t) = e^{-\pi t^2}$, then $A_a = \aaf $ is the
classical \aw\ operator and the mapping $a \to \aaf  $ is interpreted
as a quantization rule~\cite{Berezin71,Shubin91,Wong02}.

Rewriting a localization operator $\aaf$  as a Weyl operator, cf. \eqref{WA}, 
we can investigate boundedness properties for localization operators as boundedness conditions for Weyl operators having symbols $a\ast W(\f_2,\f_1)$, see \eqref{eq2}. Again it comes into play Theorem \ref{T1}.
\begin{theorem}\label{tempbound}
	Assume $s\geq0$, the indices  $p_i,q_i,p,q\in [1,\infty]$, $i=1,2$, fulfil the relations \eqref{indicitutti} and \eqref{indiceq}. Consider $r\in [1,2]$. If $a\in  M^{p,q}_{1/v_{-s}}(\Renn)$ and $\f_1,\f_2\in
	M^{r}_{v_{2s}}(\Ren)$,  then the localization operator $\gaw$ is continuous from  $M^{p_1,q_1}_{v_s}(\Ren)$  to $M^{p_2,q_2}_{v_s}(\Ren)$ with
	$$\|\gaw\|_{op}\lesssim  \|a\|_{M^{p,q}_{1/v_{-s}}}\|\f_1\|_
	{M^{r}_{2s}}\|\f_2\|_{ M^{r}_{v_{2s}}}.$$
\end{theorem}
\begin{proof}
Using Theorem \ref{T1} for $\f_1,\f_2\in
M^{r}_{v_{2s}}(\Ren)$ we obtain that $W(\f_2, \f_1)\in M^{1,\infty}_{1\otimes v_{2s}}$, for every $r\in [1,2]$. Now   the  convolution relations in Proposition \ref{mconvmp}, in the form $M^{p,q} _{1\otimes v_{-s}} \ast M^{1,\infty}_{1\otimes v_{2s}} \subseteq M^{p , q}_{1\otimes v_s}$, yield that the Weyl symbol $\sigma=a\ast W(\f_2,\f_1)$ belongs to $M^{p, q}_{1\otimes v_s} $. The result now
	follows from  Theorem \ref{Charpseudo}.
\end{proof}
\begin{remark}
(i)	The previous result extends Theorem 3.2 in \cite{CG02} and Theorem 4.11 in \cite{Toftweight} for this particular choice of weights. We observe that further extensions of Theorem \ref{tempbound} can be considered by using more general polynomial weights satisfying condition $(4.17)$  in \cite{Toftweight}.  \\
\noindent
(ii) Using the same techniques as in the proof Theorem \ref{tempbound} one can study 
conditions on symbols and window functions such that the operator $\aaf$ is in the Schatten class $S^p$, cf., e.g. \cite[Theorem 3.4]{CG02}.
\end{remark}

\section{Further applications: The Cohen Class}\label{Cohensection}
The Cohen class \eqref{Cohenclass} was  introduced by
Cohen in \cite{Cohen1} essentially to circumvent the problem of the lack of positivity of the Wigner distribution. 
Many different kinds of kernels were proposed, in particular we recall for $\tau\in [0,1]\setminus\{1/2\}$ 
the $\tau$-kernels
\begin{equation*}
\sigma _{\tau }(x,\o )=\frac{2^{d}}{|2\tau -1|^{d}}e^{2\pi i\frac{2}{%
		2\tau -1}x\o },
\end{equation*}%
which provide the  $\tau $-Wigner
distributions  $W_{\tau }(f,f)$ \cite[Proposition 5.6]{bogetal}:
\begin{equation*}
W_{\tau }(f,f)=W(f,f)\ast \sigma _{\tau }.
\end{equation*}
We recall that such distributions can be used in the definition of the $\tau$-pseudodifferential operators, see e.g., \cite{bogetal,toft1}.
Another important  kernel is the Cohen kernel $\Theta
_{\sigma }$, which yields the Born-Jordan distribution \cite{BJinv,CDNACHA2016,CGNtfa}, given by \cite[Prop. 3.4]{CGNtfa}
 \begin{equation*}
\Theta _{\sigma }(\zeta _{1},\zeta _{2})=
\begin{cases}
-2\,\mathrm{Ci}(4\pi |\zeta _{1}\zeta _{2}|),\quad (\zeta _{1},\zeta
_{2})\in \mathbb{R}^{2},\,d=1 \\ 
\mathcal{F}(\chi _{\{|s|\geq 2\}}|s|^{d-2})(\zeta _{1},\zeta _{2}),\quad
(\zeta _{1}\zeta _{2})\in {\mathbb{R}^{2d}},\,d\geq 2,%
\end{cases}
\end{equation*}
 where $\mathrm{Ci}(t)$  is the cosine integral function. It was shown in \cite[Sec. 4]{CGNtfa} that
 $\sigma _{\tau }$,  $\tau\in [0,1]\setminus\{1/2\}$, and  $\Theta _{\sigma }$ belong to the modulation space $M^{1,\infty}(\rdd)$. Inspired by this result, Theorem \ref{cohenbound} shows continuity properties for elements of the Cohen class having kernels in the modulation space $M^{1,\infty}(\rdd)$. Let us prove Theorem \ref{cohenbound}. The main ingredient will be Theorem \ref{T1}.
 \par

\begin{proof}[Proof of Theorem \protect\ref{cohenbound}.]
If $f\in M^{p_1,q_1}_{v_s}(\rd)$, with $p_1,q_1$ satisfying \eqref{Ch},  Theorem \ref{T1} gives  that the Wigner distribution is in the corresponding $M^{p,q}_{1\otimes v_{s}}(\R^{2d})$. Then the result follows by the inclusion relation $M^{p,q} _{1\otimes v_{s}} \ast M^{1,\infty} \subseteq M^{p , q}_{1\otimes v_s}$, $s\geq 0$ (see Prop. \ref{mconvmp}).
\end{proof} 

Observe that the indices' assumptions of Theorem \ref{cohenbound} coincide with those of Corollary \ref{Wignerdistr}.
Hence the continuity properties on modulation spaces of these Cohen kernels coincide with those of the Wigner distribution. In other words, the time-frequency properties of these Cohen distributions resemble those of the Wigner distribution.

\section*{Acknowledgements}

The first author was partially supported by the Italian
Local Project \textquotedblleft Analisi di Fourier per equazioni alle
derivate parziali ed operatori pseudo-differenziali", funded by the
University of Torino, 2013.

\end{document}